\UseRawInputEncoding
\documentclass[11pt]{amsart}
\usepackage{nccmath}

\usepackage{amsmath,amsthm,amssymb}
\usepackage{times}
\usepackage{enumerate}
\usepackage{color}
\newcommand{\ubar}{\underline}

\def\dfrac{\displaystyle\frac}

\newtheorem{prop}{Proposition}
\newtheorem{theo}[prop]{Theorem}
\newtheorem{lemm}[prop]{Lemma}
\newtheorem{coro}[prop]{Corollary}
\newtheorem{rmk}[prop]{Remark}

\newtheorem{defi}[prop]{Definition}

\newcommand{\be}{\begin{equation}}
\newcommand{\ee}{\end{equation}}
\newcommand{\lt}{\left}
\newcommand{\rt}{\right}
\newcommand{\goto}{\rightarrow}
\newcommand{\bn}{\bar{\nabla}}
\newcommand{\al}{\alpha}

\renewcommand{\leq}{\leqslant}
\renewcommand{\geq}{\geqslant}
\newcommand{\td}{\tilde}
\newcommand{\R}{\mathbb{R}}
\newcommand{\M}{\mathcal{M}}
\newcommand{\ka}{\kappa}

\newcommand{\s}{\sigma}
\newcommand{\ga}{\gamma}
\newcommand{\p}{\partial}

\newcommand{\lu}{\ubar{u}}
\newcommand{\uu}{\bar{u}}
\newcommand{\gas}{\gamma^*}
\newcommand{\tbus}{\tilde{\bar{u}}^*}
\newcommand{\tlus}{\tilde{\ubar{u}}^*}
\newcommand{\T}{\mathcal{T}}

\numberwithin{equation}{section}

\begin{document}
\setlength{\baselineskip}{1.2\baselineskip}

\title[Entire $\sigma_k$ curvature flow]
{Entire $\sigma_k$ curvature flow in Minkowski space}

\author{Zhizhang Wang}
\address{School of Mathematical Science, Fudan University, Shanghai, China}
\email{zzwang@fudan.edu.cn}
\author{Ling Xiao}
\address{Department of Mathematics, University of Connecticut,
Storrs, Connecticut 06269}
\email{ling.2.xiao@uconn.edu}
\thanks{2010 Mathematics Subject Classification. Primary 53C42; Secondary 35J60, 49Q10, 53C50.}
\thanks{Research of the first author is  sponsored by Natural Science  Foundation of Shanghai, No.20JC1412400,  20ZR1406600 and supported by NSFC Grants No.11871161, 12141105}

\begin{abstract}
In this paper, we study the $\sigma_k$ curvature flow of noncompact spacelike hypersurfaces in Minkowski space. We prove that if the initial hypersurface satisfies certain conditions, then the flow exists for all time. Moreover, we show that after rescaling, the flow converges to a self-expander.
\end{abstract}

\maketitle

\section{Introduction}
\label{int}
Let $\R^{n, 1}$ be the Minkowski space with the Lorentzian metric
\[ds^2=\sum\limits_{i=1}^ndx_i^2-dx_{n+1}^2.\]
In this paper, we study the $\sigma_k$ curvature flows of noncompact spacelike hypersurfaces in Minkowski space.
Spacelike hypersurfaces $\M\subset\R^{n, 1}$ have an everywhere timelike normal
field, which we assume to be future directed and to satisfy the condition
$\lt<\nu, \nu\rt> =-1.$ Such hypersurfaces can be locally expressed as the graph of a function
$u: \R^n\rightarrow\R$ satisfying $|Du(x)|< 1$ for all $x\in\R^n.$

Given an entire spacelike hypersurface $\M_0$ embedded in $\R^{n, 1},$ we let
\[X_0: \R^n\goto\R^{n, 1}\] be an embedding with $X_0(\R^n)=\M_0. $ For a given $\al\geq 1,$ we say
a family of spacelike embeddings is a solution of the $\sigma_k^{\al/k}$ curvature flow, if for each $t>0,$
$X(\R^n, t)=\M_t$ is an entire spacelike hypersurface embedded in $\R^{n, 1},$ and $X(\cdot, t)$ satisfies
\be\label{int1}
\begin{aligned}
\frac{\p X(p, t)}{\p t}&=F^\al(\kappa[\M_t](p, t))\nu\\
X(\cdot, 0)&=\M_0,
\end{aligned}
\ee
where $F^\al(\kappa[\M_t](p, t))=\sigma_k^{\al/k}(\kappa[\M_t](p, t))$ is the $\sigma_k^{\al/k}$ curvature of $\M_t$ at $X(p, t),$ $\nu$ is the future directed unit normal vector
of $\M_t$ at $X(p, t),$ and $\sigma_k$ is the $k$-th elementary symmetric polynomial, i.e.,
\[\sigma_k(\kappa)=\sum\limits_{1\leq i_1<\cdots<i_k\leq n}\kappa_{i_1}\cdots\kappa_{i_k}.\]
Since the embeddings $X(\cdot, t)$ are spacelike, the position vector of $X(\cdot, t)$ can be written as $\M_t=\{(x, u(x, t))\mid x\in\R^n\}.$
In particular, we assume $\M_0=\{(x, u_0(x))\mid x\in\R^n\}.$
After reparametrization, we can rewrite \eqref{int1} as following equation
\be\label{flow}
\left\{
\begin{aligned}
\frac{\p u}{\p t}&=F^\al\lt(\frac{1}{w}\gamma^{ik}u_{kl}\gamma^{lj}\rt)w,\\
u(x, 0)&=u_0(x),\\
\end{aligned}
\right.
\ee
where $w=\sqrt{1-|Du|^2},$ $\gamma^{ik}=\delta_{ik}+\frac{u_iu_k}{w(1+w)},$ and $u_{kl}=D^2_{x_kx_l}u$ is the ordinary Hessian of $u.$

The curvature flow problem in Euclidean space has been extensively studied in the literature. For the mean curvature flow, Huisken \cite{Hui} proved that if the initial hypersurface $\M_0$ is smooth, closed, and strictly convex, then the mean flow exists on a finite time interval $0\leq t\leq T,$ and the $\M_t$ converges to a point as $t\goto T.$ Moreover, by a suitable rescaling, it is shown that the normalized hypersurfaces converge to a sphere. A similar result has been obtained by Chow \cite{Chow} for the $n$-th root of the Gauss curvature flow.
In \cite{And}, Andrews generalized Huisken's result via the Gauss map to a large family of curvature flows including the $k$-th root of the $\sigma_k$ curvature flow.

Recall that the hyperboloid in Minkowski space is the analogue of the sphere in Euclidean space. Natural questions to ask are:
\begin{enumerate}
\item[a).] Let $\M_0$ be an entire, spacelike, strictly convex hypersurface in $\R^{n, 1},$ does there exists a solution to \eqref{int1}?
\item[b).]If the answer to part a) is ``yes'', then after rescaling, does $\M_t$ converge to a hyperboloid?
\end{enumerate}

 These questions were investigated by Andrews-Chen-Fang-McCoy in \cite{ACFMc}. Under the assumption that the initial hypersurface $\M_0$ is spacelike, co-compact, and strictly convex, Andrews-Chen-Fang-McCoy have given affirmative answers to questions a) and b). Note that we say a hypersurface is co-compact, if it is invariant under a discrete group of ambient isometries and the quotient space with respect to this group is compact. Such hypersurfaces only constitute a small collection of spacelike and strictly convex hypersurfaces. Moreover, with this assumption, the standard maximum principle can be applied directly to \eqref{int1} without worrying about the infinity (i.e., when establish a priori estimates, $\M_t$ can be treated as compact hypersurfaces without boundary). In \cite{WX212}, for $k=n$ and $\al>0,$ the authors were able to give affirmative answeres to quations a) and b) with much weaker assumptions on $\M_0$ (for details see Section 1 of \cite{WX212}).

In this paper, we assume the initial hypersurface $\M_0$ satisfying \textbf{Condition A:}
\begin{enumerate}
\item[(1).] spacelike,
\item[(2).] strictly convex,
\item[(3).] $u_0(x)-|x|\goto\varphi(x/|x|)>0$ as $|x|\goto \infty,$
\item[(4).] there exists constant $c_0, C>0$ such that
\[\tag{$\clubsuit$} c_0<\sigma_k(\kappa[\M_0])^{\al/k}< -C\lt<X, \nu\rt>\label{Cond3}.\]
\end{enumerate}
By theorem 1 in \cite{RWX20} we know there is a large collection of hypersurfaces satisfying \textbf{Condition A}. Therefore, \textbf{Condition A} is not a strong assumption.

We consider the long time existence and convergence of equation \eqref{int1} with initial hypersurfaces $\M_0$ satisfying \textbf{Condition A.}
Unlike curvature flows in Euclidean space (see \cite{Hui, Chow, And} and references therein), we show that in general the rescaling of $\M_t$ does not converge to the hyperboloid.

We want to explain \textbf{Condition A} a little bit more. It is easy to see that for any strictly convex spacelike hypersurface $\M=\{(x, u(x))\mid x\in R^n\}$ with Gauss image equals unit ball $\bar{B}_1,$ we can move $\M$ vertically such that (3) is satisfied. Condition (4) can be viewed as a growth condition on the $\sigma_k$ curvature of $\M_0,$ i.e., $\sigma_k(\kappa[\M_0])$ cannot grow too fast as $|x|\goto\infty.$ We want to point out that hypersurfaces with $\sigma_k$ curvature bounded from above and below satisfy condition (4). Assumptions on the curvature of the initial hypersurface are often needed in proving the long-time existence and convergence of non-compact curvature flows (see \cite{Aa, ACFMc, BS09, DH} for example).

Before we state our main result, we need the definition of self-expander.
\begin{defi}
A $k$-convex hypersurface $\M_u$ is called a self-expander of the flow \eqref{int1}, if it satisfies the equation
$$\sigma_k^{\frac{\al}{k}}(\kappa[\M_u])=-\lt<X,\nu\rt>.$$
Here $X=(x,u(x))$ is the position vector of $\M_u$ and $\nu$ is the future time like unit normal of $\M_u$.
 \end{defi}
 It is clear, the hyperboloid is a self-expander. Moreover, in \cite{WX22-1}, the authors have proved that for an arbitrary $\varphi\in C^2(\mathbb{S}^{n-1})$, $\varphi>0$, there exists a unique self-expander $\M_u$ such that
 $$u(x)-|x|\goto\varphi\lt(\frac{x}{|x|}\rt),\text{ as }\ \ |x|\goto\infty.$$

This work is concerned with the long time existence and convergence of \eqref{int1} for initial data $\M_0$
satisfying \textbf{Conditions A}. In particular, we prove the following theorem.
\begin{theo}\label{theo1}
Suppose $\varphi\in C^2(\mathbb{S}^{n-1}),$ $\varphi >0,$ and the initial spacelike hypersurface $\M_{u_0}$ satisfying \textbf{Conditions A}.
Then, the curvature flow \eqref{int1} admits a solution $\M_{u(x,t)}$ for all $t>0.$ Moreover, the rescaled flow
$$\td{X}=\lt(A(t)x, \frac{1}{A(t)}u(A(t)x, t)\rt), A(t)=[(1+\al)t+1]^{\frac{1}{1+\al}}$$
converges to a self-expander  $\M_{u^{\infty}}$ with the asymptotic behavior
 $$u^{\infty}(x)-|x|\goto\varphi\lt(\frac{x}{|x|}\rt),\text{ as }\ \ |x|\goto\infty.$$
\end{theo}
\begin{rmk}
In this remark, we will explain that $\varphi>0$ is only needed for technical reasons. In other words, without this assumption we can still prove the existence and convergence of the flow \eqref{int1}. When $u_0(x)-|x|\goto\varphi$ as $|x|\goto\infty$ and $\varphi$ is not positive
on $\mathbb S^{n-1},$ we can define
\[u_0^c:=u_0+c,\,\,\mbox{where $c>0$ such that $\varphi^c:=\varphi+c>0$ on $\mathbb S^{n-1}.$}\]
Applying Theorem \ref{theo1} we know there exists $u^c(x, t)$ satisfies
\[
\left\{
\begin{aligned}
\frac{\p u}{\p t}&=F^\al\lt(\frac{1}{w}\gamma^{ik}u_{kl}\gamma^{lj}\rt)w,\\
u(x, 0)&=u^c_0(x).\\
\end{aligned}
\right.
\]
Moreover, the rescaled flow $\lt(A(t)x, \frac{1}{A(t)}u^c(A(t)x, t)\rt)$ converges to a self-expander $u^{c\infty}(x)$ as $t\goto\infty,$ and
\[u^{c\infty}(x)-|x|\goto\varphi^c\lt(\frac{x}{|x|}\rt).\]
Let $u(x, t):=u^c(x, t)-c,$ then it is clear that $u(x, t)$ satisfies \eqref{flow} and
$\lt(A(t)x, \frac{1}{A(t)}\lt(u(A(t)x, t)+c\rt)\rt)$ converges to a self-expander $u^{c\infty}(x)$ as $t\goto\infty.$ In other words, when $\varphi$ is not positive,
the curvature flow \eqref{int1} still admits a solution $\M_{u(x,t)}$ for all $t>0.$ Moreover, after moving vertically, the rescaled flow still converges to a self-expander.
\end{rmk}

The organization of this paper is as follows.  In Section \ref{sap} we prove the long time existence of the approximate problem \eqref{sap1.1}.
Local $C^1$ and $C^2$ estimates are established in Section \ref{le}. Combining Section \ref{sap} with Section \ref{le}, we prove the existence
of solution to \eqref{int1} for all time $t>0$. In Section \ref{conv}, we show that after rescaling the solution of \eqref{int1}
converges to a self-expander as $t\goto\infty$.

\section{Solvability of the approximate problem}
\label{Sap}
In the rest of this paper, the constant  $C$ in \eqref{Cond3} is assumed to be $1,$ namely we assume
\begin{eqnarray}\label{Cond}
c_0<\sigma_k^{\frac{\al}{k}}(\ka[\M_{u_0}])< -\lt<X_{u_0}, \nu_{u_0}\rt>.
\end{eqnarray}
The equivalence of \eqref{Cond3} and \eqref{Cond} will be explained in the Appendix. We will also always denote $A(t)=[(1+\al)t+1]^{\frac{1}{1+\al}}.$

Applying Theorem 1 and 3 in \cite{WX22-1}, we know there exists a locally strictly convex function $\lu(x)$ such that
\[\s_k^{\frac{\al}{k}}(\ka[\M_{\lu}])=-\lt<X_{\lu}, \nu_{\lu}\rt>\,\,\mbox{and $\lu(x)-|x|\goto\varphi\lt(\frac{x}{|x|}\rt)$ as $|x|\goto\infty.$}\]
A straightforward calculation yields $A(t)\lu\lt(\frac{x}{A(t)}\rt)$ satisfying the flow equation \eqref{flow}.
Moreover, $A(t)\lu\lt(\frac{x}{A(t)}\rt)-|x|\goto A(t)\varphi\lt(\frac{x}{|x|}\rt)$ as $|x|\goto\infty.$
This inspires us to look for a locally strictly convex solution $u(x, t)$ of the initial value problem \eqref{int1}
satisfying $u(x, t)-|x|\goto A(t)\varphi\lt(\frac{x}{|x|}\rt)$ as $|x|\goto\infty.$ If such solution $u(x, t)$ exists, then by \cite{CT90}
 we obtain that its Legendre transform $u^*$ is defined in the unit ball $ \bar{B}_1\subset\R^n.$ Furthermore, by Lemma 14 of \cite{WX20} we know that $u^*$ satisfies $u^*(\xi, t)=-A(t)\varphi(\xi)$ for $\xi\in\mathbb S^{n-1}.$

In the following, we will always use $u^*$ to denote the Legendre transform of $u$ and we will also denote $F_*=\lt(\frac{\s_n}{\s_{n-k}}\rt)^\frac{1}{k}.$ In this paper, we will consider the solvability of the following problem:
\be\label{int1.1}
\lt\{
\begin{aligned}
u^*_t&=-F_*^{-\al}(w^*\gamma^*_{ik}u^*_{kl}\gamma^*_{lj})w^*\,\,&\mbox{in $B_1\times(0, \infty)$}\\
u^*(\cdot, t)&=[(1+\al)\tilde{t}]^{\frac{1}{1+\al}}\varphi^*\,\,&\mbox{on $\partial B_1\times[0, \infty)$}\\
u^*(\cdot, 0)&=u_0^*\,\,&\mbox{on $B_1\times\{0\},$}
\end{aligned}
\rt.
\ee
where $u_0^*$ is the Legendre transform of $u_0,$ $\varphi^*(\xi)=-\varphi(\xi)$ for any $\xi\in\p B_1,$ $\td{t}=t+(1+\al)^{-1},$
$w^*=\sqrt{1-|\xi|^2},$ $\gas_{ik}=\delta_{ik}-\frac{\xi_i\xi_k}{1+w^*},$ and $u^*_{kl}=\frac{\p^2u^*}{\p\xi_k\p\xi_l}.$
It is easy to see that if $u^*(\xi, t)$ solves \eqref{int1.1}, then its Legendre transform $u(x, t)$ is the unique solution of
the initial value problem \eqref{int1}.

\label{sap}
Due to the degeneracy of equation \eqref{int1.1}, in this section, we study the solvability of the approximate problem
\be\label{sap1.1}
\left\{
\begin{aligned}
(u^*_r)_t&=-F_*^{-\al}(w^*\gas_{ik}u^*_{kl}\gas_{lj})w^*\,\,&\mbox{in $B_r\times(0, T]$},\\
u^*_r(\cdot, t)&=[(1+\al)\tilde{t}]^{\frac{1}{1+\al}}u_0^*\,\, &\mbox{on $\p B_r\times[0, T],$}\\
u^*_r(\cdot, 0)&=u^*_0\,\,&\mbox{on $B_r\times\{0\}.$}
\end{aligned}
\right.
\ee

\subsection{$C^0$ estimates for $u^*_r$}
\label{sap2}
We first construct the supersolution of \eqref{flow}.
Let $\uu^*$ is the Legendre transform of $u_0.$ By the assumption \eqref{Cond}, we can see
\[F_*^{-\al}\lt(w^*\gas_{ik}\bar{u}^{*}_{kl}\gas_{lj}\rt)=\sigma_k^{\frac{\al}{k}}(\kappa[\M_{u_0}](x))<\frac{-\bar{u}^{*}}{w^*}. \]

Consider $\tbus=\lt[(1+\al)\tilde{t}\rt]^{\frac{1}{1+\al}}\bar{u}^{*},$  we want to point out that $\tbus$ is the Legendre transform of $A(t)u_0\lt(\frac{x}{A(t)}\rt).$
Then we get
\be\label{sap2.1}
\begin{aligned}
&\tbus=[(1+\al)\tilde{t}]^{-\frac{\al}{1+\al}}\bar{u}^{*}\\
&<-[(1+\al)\tilde{t}]^{-\frac{\al}{1+\al}}w^*F_*^{-\al}(w^*\gas_{ik}\bar{u}^{*}_{kl}\gas_{lj})\\
&=-F^{-\al}_*(w^*\gas_{ik}\tbus_{kl}\gas_{lj})w^*.
\end{aligned}
\ee
Therefore, by the standard maximum principle we obtain that $\tbus\leq u_r^*.$

Similar we can show
$\tlus=\lt[(1+\al)\tilde{t}\rt]^{\frac{1}{1+\al}}\lu^{*}$ is a supersolution to \eqref{sap1.1}
and $u^*_r<\tlus.$
Here $\M_{\lu}=\{(x, \lu(x))\mid x\in\R^n\}$ is a strictly convex spacelike hypersurface satisfying
\begin{eqnarray}
\sigma_k^{\frac{\al}{k}}(\kappa[\M_{\lu}](x))=-\lt<X_{\lu}, \nu_{\lu}\rt>
\end{eqnarray}
with  $\lu(x)-|x|\goto\varphi\lt(\frac{x}{|x|}\rt)$ as $|x|\goto\infty,$ and $\lu^*$ is the Legendre transform of $\lu.$ Moreover, $\tlus$ is the Legendre transform of $A(t)\lu\lt(\frac{x}{A(t)}\rt).$

We conclude
\begin{lemm}
\label{c0-lem}
Let $u_r^*$ be a solution of \eqref{sap1.1}, then $u_r^*$ satisfies
\[\td{\bar{u}}^*\leq u_r^*<\td{\lu}^*,\]
which implies
\[-C_1[(1+\al)\td{t}]^{\frac{1}{1+\al}}<u_r^*<C_0[(1+\al)\td{t}]^{\frac{1}{1+\al}},\]
where $C_0,$ $C_1$ are positive constants depending on $\varphi^*$ and $\M_0.$
\end{lemm}

\subsection{The bounds for $F_*$}
\label{sap3}
In this subsection we will show that along the flow, $F_*$ is bounded from above and below.

We take the hyperplane $\mathbb{P}:=\{X=(x_1, \cdots, x_{n}, x_{n+1}) |\, x_{n+1}=1\}$ and consider the projection of
$\mathbb{H}^n(-1)$ from the origin into $\mathbb{P}.$ Then $\mathbb{H}^n(-1)$ is mapped in
a one-to-one fashion onto an open unit ball $B_1:=\{\xi\in\R^n |\, \sum\xi^2_k<1\}.$ The map
$P$ is given by
\[P: \mathbb{H}^n(-1)\rightarrow B_1;\,\,(x_1, \cdots, x_{n+1})\mapsto (\xi_1, \cdots, \xi_n),\]
where $x_{n+1}=\sqrt{1+x_1^2+\cdots+x_n^2},$ $\xi_i=\frac{x_i}{x_{n+1}}.$
When $u^{*}_r$ satisfies \eqref{sap1.1}, let $v_r=\frac{u^{*}_r}{w^*}$ then a straight forward calculation yields $v_r$ satisfies
\be\label{sap3.1}
\left\{
\begin{aligned}
(v_r)_t&=-\tilde{F}^{-1}(\Lambda_{ij}):=\tilde{G}(\Lambda_{ij})\,\,&\mbox{in $P^{-1}(B_r)\times(0, T]:=U_r\times(0, T]$}\\
v_r(\cdot, t)&=\lt[(1+\al)\tilde{t}\rt]^{\frac{1}{1+\al}}\frac{u_0^*}{\sqrt{1-r^2}}\,\,&\mbox{on $\p U_r\times[0, T]$}\\
v_r(\cdot, 0)&=u_0^*x_{n+1}\,\,&\mbox{on $U_r\times\{0\},$}
\end{aligned}
\right.
\ee
where $\tilde{F}=F_*^\al,$ $\Lambda_{ij}=\bn_{ij}v_r-v_r\delta_{ij},$ and $\bn$ is the Levi-Civita Connection of the hyperbolic space.

\begin{lemm}
\label{sap-lem3.1}
Assume  $v$ is a solution of \eqref{sap3.1}. Then we have $$\frac{[(1+\al)T+1]^{\frac{\al}{1+\al}}}{C_3}>\tilde{F}>\frac{1}{C_2x_{n+1}}\,\, \mbox{on $\bar{U}_r\times(0, T]$.}$$
Here,  $C_2=C_2(|u_0^*|_{C^0})$ and $C_3=C_3(c_0, |u_0^*|, r),$ where  $c_0>0$ is the lower bound of $\sigma_k^{\frac{\al}{k}}(\ka[\M_0])$ in \eqref{Cond}.
\end{lemm}
\begin{proof}
Since
\[\tilde{G}_t=\tilde{G}^{ij}((v_t)_{ij}-v_t\delta_{ij})=\tilde{G}^{ij}(\bn_{ij}\tilde{G}-\tilde{G}\delta_{ij}),\]
we have $\mathcal{L}\tilde{G}=-\tilde{G}\sum_{i}\tilde{G^{ii}},$ where $\mathcal{L}:=\frac{\p}{\p t}-\tilde{G}^{ij}\bn_{ij}.$
It is clear that $\mathcal{L}x_{n+1}=-x_{n+1}\sum_i\tilde{G}^{ii}$.
Therefore, we get
$$\mathcal{L}\frac{-\tilde{G}}{x_{n+1}}=\frac{\mathcal{L}(-\tilde{G})}{x_{n+1}}+\frac{\tilde{G}\mathcal{L}(x_{n+1})}{x_{n+1}^2}+2\tilde{G}^{ij}\frac{(x_{n+1})_i}{x_{n+1}}\left(\frac{-\tilde{G}}{x_{n+1}}\right)_j=2\tilde{G}^{ij}\frac{(x_{n+1})_i}{x_{n+1}}\left(\frac{-\tilde{G}}{x_{n+1}}\right)_j.$$
Applying the maximum principal, we get $\tilde{F}x_{n+1}$ achieves its maximum and minimum at the parabolic boundary.
Now denote $\tilde{v}_r=v_r-\lt[(1+\al)\tilde{t}\rt]^{\frac{1}{1+\al}}\frac{u_0^*}{\sqrt{1-r^2}}$, then $\tilde{v}_r$ satisfies
\[
\left\{
\begin{aligned}
(\tilde{v}_r)_t&=-\tilde{F}^{-1}-\lt[(1+\al)\tilde{t}\rt]^{\frac{-\al}{1+\al}}\frac{u_0^*}{\sqrt{1-r^2}}\,\,&\mbox{in $U_r\times(0, T],$}\\
v_r(\cdot, t)&=0\,\,&\mbox{on $\p U_r\times[0, T],$}\\
v_r(\cdot, 0)&=u_0^*x_{n+1}-\frac{u_0^*}{\sqrt{1-r^2}}\,\,&\mbox{on $U_r\times\{0\}.$}
\end{aligned}
\right.
\]
In view of the short time existence theorem, we obtain on $\p U_r\times (0, T]$
\[\tilde{F}=\frac{\lt[(1+\al)\tilde{t}\rt]^{\frac{\al}{1+\al}}\sqrt{1-r^2}}{-u^*_0}.\]
Therefore, we get, on $\p U_r\times(0,T]$, $$\tilde{F}x_{n+1}=\frac{\lt[(1+\al)\tilde{t}\rt]^{\frac{\al}{1+\al}}}{-u^*_0};$$
on $U_r\times\{0\}$, by \eqref{Cond}, we have
$$c_0x^{-1}_{n+1}<\tilde{F}^{-1}x_{n+1}^{-1}<|u_0^*|.$$
This completes the proof of Lemma \ref{sap-lem3.1}.
\end{proof}

\subsection{$C^1$ estimates for $u_r^*$}
\label{sap4}
\begin{lemm}
\label{sap-lem4.1}
Let $u^*_r$ be a solution of \eqref{sap1.1}, then
$|Du_r^*|\leq C:= C(\M_0, r, t, T)\,\,\mbox{ in $\bar{B}_r \times[0, T].$}$
\end{lemm}
\begin{proof}
By Lemma \ref{sap-lem3.1} we know that, for any $t\in(0, T]$  we have
\[C_1<\frac{\s_n}{\s_{n-k}}(w^*\gas_{ik}(u^*_r)_{kl}\gas_{lj})<C_2\]
for some $C_1, C_2$ depending on $\M_0, r, T.$
Applying Section 5 of \cite{WX20}, we get, for any fixed $t\in(0, T]$ there exists $\uu^{*t}$ and $\lu^{*t}$ such that
\[
\left\{
\begin{aligned}
\frac{\s_n}{\s_{n-k}}(w^*\gas_{ik}(\uu^{*t})_{kl}\gas_{lj})&=C_1\,\,&\mbox{in $B_r$}\\
\uu^{*t}&=[(1+\al)\tilde{t}]^{\frac{1}{1+\al}}u^*_0\,\,&\mbox{on $\p B_r,$}
\end{aligned}
\right.
\]
and
\[
\left\{
\begin{aligned}
\frac{\s_n}{\s_{n-k}}(w^*\gas_{ik}(\lu^{*t})_{kl}\gas_{lj})&=C_2\,\,&\mbox{in $B_r$}\\
\lu^{*t}&=[(1+\al)\tilde{t}]^{\frac{1}{1+\al}}u^*_0\,\,&\mbox{on $\p B_r,$}
\end{aligned}
\right.
\]
where $\tilde{t}=t+(1+\al)^{-1}.$
By the maximum principle we derive
\[\lu^{*t}\leq u^*(\cdot, t)\leq\uu^{*t}\,\,\mbox{in $\bar{B}_r.$}\]
Therefore, we obtain
\[|Du^*_r|\leq\max\limits_{\xi\in\p B_r}\{|D\lu^{*t}|, |D\uu^{*t}|\},\]
here we have used the convexity of $u^*_r$.
\end{proof}

\subsection{$C^2$ boundary estimates for $u^*_r$}
\label{sap5}
In this subsection, we will show $D^2u^*_r$ is bounded on $\p B_r\times[0, T].$ From now on, for our convenience, we will consider the solution of equation \eqref{sap3.1}. We suppose $\{\tau_1,\cdots,\tau_n\}$ is the orthonormal frame of the boundary $\p U_r$. Moreover, $\tau_1,\cdots,\tau_{n-1}$ are the tangential vectors and $\tau_n$ is the unit interior normal vector.
 \begin{lemm}
\label{sap-lem5.1}
Let $v$ be the solution of \eqref{sap3.1}, then the second order tangential derivatives on the boundary satisfy
$|\bn_{\al\beta} v|\leq C$ on $\p U_r\times (0, T]$ for $\al, \beta<n.$ Here $C$ depends on $\M_0, r,T$.
\end{lemm}
\begin{proof}
Recall the calculation in Subsection \ref{sap2} we know that $\tbus$ is a subsolution of \eqref{sap3.1}.
Let
\begin{eqnarray}\label{eq2}
\ubar{v}=\frac{\tbus}{\sqrt{1-|\xi|^2}}.
\end{eqnarray}
 Since $v-\ubar{v}\equiv 0$ on $\p U_r\times[0, T],$ we have
\[\bn_{\al\beta}(v-\ubar{v})=-\bn_{n}(v-\ubar{v})II(\tau_\al, \tau_\beta)\]
on $\p U_r\times (0, T].$ Here $II$ is the second fundamental form of $\p U_r$. Applying Lemma \ref{sap-lem4.1}, we complete the proof of this Lemma.
\end{proof}

Now, we will show that $|\bn_{\al n}v|$ is bounded. In the following we will denote
\be\label{operator-L}
\mathfrak{L}\phi:=\phi_t-\tilde{F}_v^{-2}\tilde{F}_v^{ij}\bn_{ij}\phi+\phi\tilde{F}_v^{-2}\sum_i\tilde{F}^{ii}_v
\ee
for any smooth function $\phi.$ Here $\tilde{F}_v(\Lambda_{ij})=F^\al_*(\Lambda_{ij}),$ $\tilde{F}^{ij}_v=\frac{\p\tilde{F}_v}{\p\Lambda_{ij}},$
and $\Lambda_{ij}=\bn_{ij}v-v\delta_{ij}.$

\begin{lemm}
\label{sap-lem5.2}
Let $v$ be a solution of \eqref{sap3.1}, $\ubar{v}$ be the subsolution of \eqref{sap3.1} which is defined by \eqref{eq2}, and
$h=(v-\ubar{v})+B\lt(\frac{1}{\sqrt{1-r^2}}-x_{n+1}\rt),$ where $B>0$ is a constant. Then for any given constant $B_1>0$, there exists a sufficiently large $B$ depending on $\M_0, r, T$, such that $\mathfrak{L}h>\frac{B_1}{\tilde{F}_v^2}\sum\tilde{F}^{ii}_v.$
\end{lemm}
\begin{proof}
A direct calculation gives
\[
\begin{aligned}
&\mathfrak{L}\lt(\frac{1}{\sqrt{1-r^2}}-x_{n+1}\rt)\\
=&-\tilde{F}_v^{-2}\tilde{F}_v^{ij}\bn_{ij}(-x_{n+1})+\lt(\frac{1}{\sqrt{1-r^2}}-x_{n+1}\rt)\tilde{F}_v^{-2}\sum\tilde{F}_v^{ii}\\
=&\frac{1}{\sqrt{1-r^2}}\tilde{F}^{-2}_v\sum\tilde{F}^{ii}_v,
\end{aligned}
\]
here we have used $\bn_{ij}x_{n+1}=x_{n+1}\delta_{ij}$.
It is easy to see that to prove Lemma \ref{sap-lem5.2} we only need to show there exists some $B_2>0$ such that
\be\label{sap5.2}
\mathfrak{L}(v-\ubar{v})>-\lt(\ubar{v}_t+\tilde{F}^{-1}(\ubar{\Lambda}_{ij})+\frac{B_2}{\tilde{F}^2_v}\sum_i\tilde{F}^{ii}_v\rt),
\ee
where $\ubar{\Lambda}_{ij}=\bn_{ij}\ubar{v}-\ubar{v}\delta_{ij}.$

In the rest of this proof, the value of $B_2$ may vary from line to line.
Notice that inequality \eqref{sap5.2} is equivalent to
\[
\begin{aligned}
&(v-\ubar{v})_t-\tilde{F}^{-2}_v\tilde{F}_v^{ij}\bn_{ij}(v-\ubar{v})_{ij}+(v-\ubar{v})\tilde{F}^{-2}_v\sum_i\tilde{F}^{ii}_v\\
>&-\ubar{v}_t-\tilde{F}^{-1}(\ubar{\Lambda}_{ij})-\frac{B_2}{\tilde{F}^2_v}\sum_i\tilde{F}_v^{ii},\\
\end{aligned}
\]
which can be simplified as
\be\label{sap5.3}
-\frac{(1+\al)}{\tilde{F}_v}+\tilde{F}_v^{-2}\tilde{F}^{ii}_v\ubar{\Lambda}_{ii}+\frac{B_2}{\tilde{F}^2_v}\sum_i\tilde{F}^{ii}_v>-\tilde{F}^{-1}(\ubar{\Lambda}_{ij}).
\ee
Since $\tilde{F}^{\frac{1}{\al}}$ is concave we have
\[\frac{1}{\al}\tilde{F}_v^{\frac{1}{\al}-1}\tilde{F}_v^{ij}\ubar{\Lambda}_{ij}\geq\tilde{F}^{\frac{1}{\al}}(\ubar{\Lambda}_{ij})\ \
\text{ and } \ \frac{1}{\al}\tilde{F}_v^{\frac{1}{\al}-1}\sum_i\tilde{F}^{ii}_v\geq \td{F}^{\frac{1}{\al}}(1, \cdots, 1).\] This implies
\[\mbox{\textbf{l.h.s.} of \eqref{sap5.3}$\geq-\dfrac{1+\al}{\tilde{F}_v}+\tilde{F}_v^{-2}
\lt[\dfrac{\al\tilde{F}^{\frac{1}{\al}}(\ubar{\Lambda}_{ij})}{\tilde{F}^{\frac{1}{\al}-1}_v}\rt]+\dfrac{\al B_2}{\tilde{F}_v^{\frac{1}{\al}+1}c},$}\]
where $c=(C^k_n)^{1/k}$ is the combination number. One can see that \eqref{sap5.3} can be derived from the following inequality:
\begin{eqnarray}\label{eq3}
\al\tilde{F}^{\frac{1}{\al}}(\ubar{\Lambda}_{ij})+\frac{B_2\al}{c}+\frac{\tilde{F}_v^{1+\frac{1}{\al}}}{\tilde{F}(\ubar{\Lambda}_{ij})}\geq(1+\al)\tilde{F}_v^{\frac{1}{\al}}.
\end{eqnarray}
Recall \eqref{eq2}, we have
\[\tilde{F}(\ubar{\Lambda}_{ij})=F_*^\al(w^*\gas_{ik}\tbus_{kl}\gas_{lj})<[(1+\al)\tilde{t}]^\frac{\al}{1+\al}c_0^{-1},\]
where $c_0$ is given by assumption \eqref{Cond}.
When $\tilde{F}_v>[(1+\al)\tilde{T}]^\frac{\al}{1+\al}c_0^{-1}(1+\al),$  we have $\tilde{F}_v^{1+\frac{1}{\al}}\geq(1+\al)\tilde{F}_v^{\frac{1}{\al}}$ and \eqref{eq3} follows.
When $\tilde{F}_v\leq [(1+\al)\tilde{T}]^\frac{\al}{1+\al}c_0^{-1}(1+\al),$ choose
$B_2\geq c\frac{(1+\al)^{1+\frac{1}{\al}}}{\al}[(1+\al)\tilde{T}]^{\frac{1}{(1+\al)}}c_0^{-\frac{1}{\al}},$ then \eqref{eq3} follows.
Here $\tilde{T}=T+(1+\al)^{-1}.$
Therefore, we conclude that when $B_2:=B_2(n, k, c_0, T)$ is sufficiently large, \eqref{eq3} always holds. This completes the proof of this Lemma.
\end{proof}

\begin{lemm}
\label{sap-lem5.3}
Let $v$ be a solution of \eqref{sap3.1} and suppose $\tau_n$ is the interior unit normal vector filed of $\p U_r.$ We have $|\bn_{\al n}v|\leq C$
on $\p U_r\times (0, T]$. Here $C$ depends on $\M_0, r, T$.
\end{lemm}
\begin{proof}
Denote the angular derivative $\T:=\xi_\al\p_n-\xi_n\p_\al,$ where $\p_i:=\frac{\p}{\p\xi_i}$
and $\al<n.$ Let $\phi(\cdot, t)=u_r^*(\cdot, t)-\tbus(\cdot, t),$
then on $\lt(\p B_r\times[0, T]\rt)\cup \lt(B_r\times\{0\}\rt)$ we have $\phi\equiv 0.$ Therefore, we have
\[\T\phi(x, t)=0\,\,\mbox{for $(x, t)\in\p B_r\times[0, T]$ and $\T\phi(x, 0)=0$ for $x\in B_r$}.\]
In the following we denote
$\tilde{G}(w^*\gas_{ik}(u_r^*)_{kl}\gas_{lj})=-F_*^{-\al}(\ka^*[w^*\gas_{ik}(u_r^*)_{kl}\gas_{lj}]),$ then $u_r^*$ satisfies
\be\label{sap5.4}
\frac{\p u_r^*}{\p t}-\tilde{G}(w^*\gas_{ik}(u^*_r)_{kl}\gas_{lj})w^*=0.
\ee
Notice that  $\T$ is an angular derivative vector with respect to the origin of $B_r$.
By Lemma 29 of \cite{RWX19}, we get
\be\label{sap5.5}\frac{\p (\T u_r^*)}{\p t}-\tilde{G}^{ij}(w^*\gas_{ik}(\T u^*_r)_{kl}\gas_{lj})w^*=0,\ee
where  $\tilde{G}^{ij}=\frac{\p\tilde{G}}{\p a_{ij}}$ for  $a_{ij}=w^*\gas_{ik}(u^*_r)_{kl}\gas_{lj}$.

Recall that $\tbus=[(1+\al)\tilde{t}]^{\frac{1}{1+\al}}\uu^*$ and $\uu^*$ is the Legendre transform of $u_0.$ It is clear that $\T\tbus=[(1+\al)\tilde{t}]^\frac{1}{1+\al}\T\uu^*.$
Applying Lemma 15 of \cite{WX20}, we know
\[\bn_{ij}\lt(\frac{u^*}{w^*}\rt)-\frac{u^*}{w^*}\delta_{ij}=w^*\gas_{ik}u^*_{kl}\gas_{lj}.\]
Combining with \eqref{sap5.5} we obtain
\[(\T v)_t-\tilde{G}^{ij}\bn_{ij}(\T v)+(\T v)\sum\tilde{G}^{ii}=0,\]
where $\T v=\T\lt(\frac{u_r^*}{w^*}\rt)=\frac{\T u^*_r}{w^*},$ and $\tilde{G}^{ij}=\tilde{F}^{-2}_v\tilde{F}^{ij}_v$. A straightforward calculation yields
\[
\begin{aligned}
&|\mathfrak{L}\T\ubar{v}|\\
=&\bigg|[(1+\al)\tilde{t}]^{-\frac{\al}{1+\al}}\T \frac{\uu^*}{w^*}-\tilde{G}^{ij}\bn_{ij}\left(\T \frac{\uu^*}{w^*}\right)\cdot[(1+\al)\tilde{t}]^{\frac{1}{1+\al}}+[(1+\al)\tilde{t}]^{\frac{1}{1+\al}}\T \frac{\uu^*}{w^*}\sum_i\tilde{G}^{ii}\bigg|\\
\leq& C_4\sum_i\tilde{G}^{ii},
\end{aligned}
\]
where $C_4$ is some constant depending on $\uu^*, T, r$.

Let $\tilde{\psi}(x,t)=\T v(x,t)-\T\ubar{v}(x,t),$ then we have
\[
\begin{aligned}
\tilde{\psi}(x, 0)= 0\,\,\mbox{for any $x\in U_r,$}\ \
\tilde{\psi}(x, t)= 0\,\, \mbox{for any $(x,t)\in\p U_r\times (0, T].$}
\end{aligned}
\]
Moreover,  we have
$\lt|\mathfrak{L}\tilde{\psi}\rt|\leq C_4\sum\limits_i\tilde{G}^{ii}.$ Applying Lemma \ref{sap-lem5.2} we have, when $B>0$
 very large
 \[\mathfrak{L}(\tilde{\psi}-h)\leq0\,\,\mbox{in $U_r\times(0, T].$}\]
 It's easy to see that $\tilde{\psi}\leq h$ on the parabolic boundary $(\p U_r\times(0, T])\cup(U_r\times\{0\}).$ By the maximum principle, we get
 $h\geq\tilde{\psi}$ in $U_r\times(0, T].$ Therefore, $h_n>\tilde{\psi}_n$ on $\p U_r\times(0, T],$
 which yields $\bn_{\al n}v\leq C_5.$ Similarly, by considering $\T\ubar{v}-\T v$, we obtain
 $\bn_{\al n}v\geq C_{6}.$
 \end{proof}

 \begin{lemm}\label{sap-lem5.4}
 Let $v$ be a solution of \eqref{sap3.1} and suppose $\tau_n$ is the interior unit normal vector filed of $\p U_r.$ Then we have $|\bn_{nn}v|\leq C$
on $\p U_r\times (0, T],$ where  $C$ is a positive constant depending on $\M_0, r, T$.
 \end{lemm}
\begin{proof}
Recall that in the proof of Lemma \ref{sap-lem3.1} we have shown, on the boundary $\p U_r\times (0, T]$,
\[F_*^\al=\frac{[(1+\al)\tilde{t}]^{\frac{\al}{1+\al}}\sqrt{1-r^2}}{-u_0^*},\]
where $F_*=\lt(\frac{\s_n}{\s_{n-k}}\rt)^{\frac{1}{k}}.$ This is equivalent to,  on $\p U_r\times(0, T],$
\[\frac{\s_n}{\s_{n-k}}=[(1+\al)\tilde{t}]^{\frac{k}{1+\al}}\lt(\frac{\sqrt{1-r^2}}{-u_0^*}\rt)^{\frac{k}{\al}}.\]
We will adapt the idea of \cite{tru} to prove this Lemma.
Recall the formula (2.5) in \cite{tru}, we have
\[
\begin{aligned}
\s_k(\ka[\Lambda_{ij}])&=\s_{k-1}(\ka[\Lambda_{\al\beta}])
\Lambda_{nn}+\s_k(\ka[\Lambda_{\al\beta}])-\sum\limits_{\gamma=1}^{n-1}\s_{k-2}(\ka[\Lambda_{\al\beta}|\Lambda_{\gamma\gamma}])\Lambda^2_{\gamma n},
\end{aligned}
\]
where $1\leq i, j\leq n$ and $1\leq\al, \beta, \gamma\leq n-1.$
We define
$$A_{k-1}=\s_{k-1}(\ka[\Lambda_{\al\beta}]) \text{ and } \ \ B_k=\s_k(\ka[\Lambda_{\al\beta}])-\sum\limits_{\gamma=1}^{n-1}\s_{k-2}(\ka[\Lambda_{\al\beta}|\Lambda_{\gamma\gamma}])\Lambda^2_{\gamma n}.
$$
Therefore, we obtain on $\p U_r\times(0, T]$
\begin{eqnarray}\label{eq10}
\frac{\s_n}{\s_{n-k}}=\frac{A_{n-1}\Lambda_{nn}+B_n}{A_{n-k-1}\Lambda_{nn}+B_{n-k}}
=[(1+\al)\tilde{t}]^{\frac{k}{1+\al}}\lt(\frac{\sqrt{1-r^2}}{-u_0^*}\rt)^{\frac{k}{\al}}.
\end{eqnarray}
Now, we prove
\be\label{sap5.6}
\frac{A_{n-1}}{A_{n-k-1}}\geq\frac{\s_n}{\s_{n-k}}=\frac{A_{n-1}\Lambda_{nn}+B_n}{A_{n-k-1}\Lambda_{nn}+B_{n-k}},
\ee
which is
equivalent to the following inequality
\[
\begin{aligned}
\s_{n-1}(\ka')\lt[\s_{n-k}(\ka')-\sum\limits_{\gamma=1}^{n-1}\s_{n-k-2}(\ka'|\gamma)\Lambda^2_{\gamma n}\rt]\geq\s_{n-k-1}(\ka')\lt[-\sum\limits_{\gamma=1}^{n-1}\s_{n-2}(\ka'|\gamma)\Lambda^2_{\gamma n}\rt],\\
\end{aligned}
\]
where $\ka'=\ka'[\Lambda_{\al\beta}]$
are eigenvalues of $(\Lambda_{\al\beta}).$
For any fixed $\gamma$, we have
\[
\begin{aligned}
\s_{n-k-1}(\ka')\s_{n-2}(\ka'|\gamma)&=[\s_{n-k-1}(\ka'|\gamma)+\s_{n-k-2}(\ka'|\gamma)\ka_\gamma]\s_{n-2}(\ka'|\gamma)\\
&>\s_{n-1}(\ka')\s_{n-k-2}(\ka'|\gamma),
\end{aligned}
\]
which gives \eqref{sap5.6} directly.

In the rest of this proof, we will denote $f=\lt(\frac{\s_{n-1}(\ka')}{\s_{n-k-1}(\ka')}\rt)^{\frac{1}{k}}.$ In view of \eqref{sap5.6} we get, on $\p U_r\times(0, T]$
\[f\geq[(1+\al)\tilde{t}]^{\frac{1}{1+\al}}\lt(\frac{\sqrt{1-r^2}}{-u_0^*}\rt)^{\frac{1}{\al}}.\]
Now consider the following function defined on $\p U_r\times[0, T],$
\[d(x, t):=f-[(1+\al)\tilde{t}]^{\frac{1}{1+\al}}\lt(\frac{\sqrt{1-r^2}}{-u_0^*}\rt)^{\frac{1}{\al}}.\]
By our assumption \eqref{Cond} on $u_0$, we derive
\[\frac{\s_n}{\s_{n-k}}(w^*\gas_{ik}(u_0^*(\xi))_{kl}\gas_{lj})>\lt(\frac{\sqrt{1-r^2}}{-u_0^*(\xi)}\rt)^{\frac{k}{\al}}
\,\,\mbox{for $\xi\in\p B_r,$}\]
which implies $d(x, 0)\geq c>0$ on $\p U_r\times\{0\}.$
Therefore, we may  assume
$d(x, t)$ achieves its minimum at the point $(y, t_0)\in\p U_r\times(0, T].$ If not, we would have $d\geq c$ on $\p U_r\times [0,T].$ Then in view of \eqref{sap5.5} we would derive that $\Lambda_{nn}$ has an uniform upper bound on $\p U_r\times [0,T].$ The Lemma would follow easily.

 Note that at any $(x, t)\in\p U_r\times(0, T]$, for $\al,\beta<n$, we have
\be\label{sap5.7}
\bn_{\al\beta}(v-\ubar{v})=-\bn_n(v-\ubar{v})\rho_{\al\beta},
\ee
where $\rho_{\al\beta}$ is the second fundamental form of $\p U_r$ with respect the $\tau_n$.
For a symmetric matrix $r=(r_{\al\beta})$ with eigenvalues $\lambda_1,\cdots, \lambda_{n-1},$
let's define $G(r)=f(\lambda_1, \cdots,\lambda_{n-1})=\lt(\frac{\s_{n-1}(\lambda)}{\s_{n-k-1}(\lambda)}\rt)^{\frac{1}{k}},$
$G^{\al\beta}=\frac{\p G}{\p r_{\al\beta}},$ and $G_0^{\al\beta}=G^{\al\beta}\big|_{\Lambda_{\al\beta}(y, t_0)}.$

By the concavity of $f$, we have
\[G_0^{\al\beta}\Lambda_{\al\beta}(x, t)\geq G(\Lambda_{\al\beta}(x, t)),\,\, \text{and}\,\,  G_0^{\al\beta}\ubar{\Lambda}_{\al\beta}(x, t)\geq G(\ubar{\Lambda}_{\al\beta}(x, t)).\]
Therefore, for all $(x, t)\in\p U_r\times[0, T]$ we have
\[G^{\al\beta}_0\Lambda_{\al\beta}(x, t)-[(1+\al)\tilde{t}]^{\frac{1}{1+\al}}\lt(\frac{\sqrt{1-r^2}}{-u_0^*}\rt)^\frac{1}{\al}
\geq d(x, t)\geq d(y, t_0).\]
When $(x, t)\in\p U_r\times[0, T],$ using the above inequity we get,
\be\label{sap5.8}
\begin{aligned}
&G_0^{\al\beta}\bn_n\lt(v-\ubar{v}\rt)\rho_{\al\beta}\\
=&G^{\al\beta}_0\bn_{\al\beta}\lt[\ubar{v}(x, t)-v(x, t)\rt]\\
=&G_0^{\al\beta}[\bn_{\al\beta}\ubar{v}(x, t)-\ubar{v}(x, t)\delta_{\al\beta}]-G_0^{\al\beta}[\bn_{\al\beta}v(x, t)-v(x, t)\delta_{\al\beta}]\\
\leq& G_0^{\al\beta}\ubar{\Lambda}_{\al\beta}(x, t)-d(y, t_0)-[(1+\al)\tilde{t}]^{\frac{1}{1+\al}}\lt(\frac{\sqrt{1-r^2}}{-u_0^*}\rt)^{\frac{1}{\al}}.
\end{aligned}
\ee

Suppose $U_{r\delta}:=\lt\{x\in U_r\big|\frac{1}{\sqrt{1-r^2}}-x_{n+1}<\delta\rt\}$ and $\Gamma_\theta:=\lt\{x\in U_r\big|\frac{1}{\sqrt{1-r^2}}-x_{n+1}=\theta\rt\}$ for $0\leq \theta\leq \delta$. Now we extend  $\tau_1,\tau_2,\cdots,\tau_{n}$ to $U_{r\delta}$ such that it is still an orthonormal frame and the first $n-1$ vectors are the tangential vectors of $\Gamma_{\theta}$, $\tau_n$ is the interior normal of $\Gamma_{\theta}$. Therefore, we can extend $\rho_{\alpha\beta}$ to be the second fundamental form of $\Gamma_{\theta}$. Thus, it is clear that there exists $c_1>0$ such that
\[G_0^{\al\beta}\rho_{\al\beta}(x)\geq c_1\,\,\mbox{in $\bar{U}_{r\delta}\times[0, T]$.}\]
By \eqref{sap5.8}, we get, on $\p U_r\times[0,T]$,
\[\bn_n(v-\ubar{v})\leq\frac{G_0^{\al\beta}\ubar{\Lambda}_{\al\beta}(x, t)-d(y, t_0)-[(1+\al)\tilde{t}]^{\frac{1}{1+\al}}\lt(\frac{\sqrt{1-r^2}}{-u_0^*}\rt)^\frac{1}{\al}}
{G_0^{\al\beta}\rho_{\al\beta}(x)}.\]
We denote the right hand side of the above inequality by $\psi(x, t).$  Then we define
$$\phi:=\bn_n(v-\ubar{v})-\psi-C\left(\frac{1}{\sqrt{1-r^2}}-x_{n+1}\right) \,\,\mbox{ on }\bar{U}_{r\delta}\times[0, T].$$
It is obvious that $\phi(y, t_0)=0,$ and $\phi\leq 0$ on $\p U_r\times[0,T]$.
We use the coordinate $x_1,x_2,\cdots,x_n$ for the hyperboloid. Thus, $U_r$ is the $n-1$-dimensional ball $\mathcal B_{n-1}\subset\mathbb{H}^n$ defined by$|x|\leq\frac{r}{\sqrt{1-r^2}}$. For any $x\in \mathcal B_{n-1}$, we suppose $p$ is the point on $\p \mathcal B_{n-1}$ such that
$\text{dist}(x, \p\mathcal B_{n-1})=\text{dist}(x, p)$. Here, $\text{dist}(\cdot, \cdot)$ is the Euclidean distance function. We further let
$$\tilde{\psi}=\psi-\bn_n(v-\ubar{v}).$$
Thus, we get
$$-\tilde{\psi}(x,0)\leq\tilde{\psi}(p,0)-\tilde{\psi}(x,0)\leq |D\tilde{\psi}(\cdot,0)||x-p|\leq C_0(r_0-|x|),$$ where $r_0=\frac{r}{\sqrt{1-r^2}}$ and $C_0>0$ is a constant depending on $U_r, \M_0$. On the other hand, we have
$$\frac{1}{\sqrt{1-r^2}}-x_{n+1}=\frac{(r_0-|x|)(r_0+|x|)}{\frac{1}{\sqrt{1-r^2}}+x_{n+1}}\geq \frac{r}{2}(r_0-|x|).$$
Therefore, we can choose a sufficiently large $C>0$  such that on $U_r\times\{0\},$ $\phi \leq 0;$ and
on $\Gamma_{\delta}\times[0,T],$ $\phi<0.$

Differentiating the flow equation $v_t+\tilde{F}^{-1}(v_{ij}-v\delta_{ij})=0$ with respect to $\tau_n$, we get
\[v_{nt}-\tilde{F}_v^{-2}F_v^{ij}(v_{nij}-v_n\delta_{ij})=0.\]
This yields $\lt|\mathfrak{L}\phi\rt|\leq C_7\tilde{F}_v^{-2}\sum_i\tilde{F}_v^{ii},$ where $C_7$ depends on $r, \M_0,$ and $T.$
Therefore, if we choose sufficiently large $B$ in the definition of $h$, we conclude
\[\phi\leq h\,\,\mbox{in $\bar{U}_r\times[0, T],$}\]
which implies $\bn_{nn}v(y, t_0)<C_8.$ Thus, there exists a constant $C_9>0$ such that $d(x, t)\geq C_9$ on $\p U_r\times[0, T],$ which in turn gives
$\bn_{nn}v(x, t)<C_10$ on $\p U_r\times [0, T].$ Thus we prove that $\bn_{nn}v$ is bounded from above on the parabolic boundary. The lower bound for
$\bn_{nn}v$ comes from the matrix $(\Lambda_{ij})$ is positive definite. This completes the proof of this Lemma.
\end{proof}

The global $C^2$ estimates for $u_r^*$- solution of \eqref{sap1.1} follows from Lemma 20 in \cite{WX212} directly. Therefore, we have proved the solvability of the approximate problem
\eqref{sap1.1}.

\section{Local estimates for $u_r$}
\label{le}
In this section, we will establish local estimates for $u_r,$ the Legendre transform of  $u^*_r$-solution of \eqref{sap1.1}.
\subsection{Local $C^0$ estimates for $u^*_r$}
\label{lc2}
By Lemma \ref{c0-lem}, we have $\tbus<u_r^*<\tlus,$ where $\tbus=[(1+\al)\tilde{t}]^{\frac{1}{1+\al}}\bar{u}^{*}$
and $\tlus=[(1+\al)\tilde{t}]^{\frac{1}{1+\al}}\ubar{u}^{*}.$ By Lemma 13 in \cite{WX20}, we get
\[[(1+\al)\tilde{t}]^{\frac{1}{1+\al}}\ubar{u}\lt(\frac{x}{[(1+\al)\tilde{t}]^{\frac{1}{1+\al}}}\rt)
<u_r(x,t)<[(1+\al)\tilde{t}]^{\frac{1}{1+\al}}\bar{u}\lt(\frac{x}{[(1+\al)\tilde{t}]^{\frac{1}{1+\al}}}\rt),\]
which is a local $C^0$ estimate for $u_r$.

\subsection{Local $C^1$ estimates}
\label{sub-loc-c1}
We introduce a new subsolution $\lu_1$ satisfying
\[\s_k^{\frac{\al}{k}}(\ka[\M_{\lu_1}(x)]=-10\lt<X_{\lu_1}, \nu_{\lu_1}\rt>\]
and as $|x|\goto\infty$
$$\lu_1\goto |x|+\varphi\left(\frac{x}{|x|}\right).$$
By the strong maximum principle we have, when $x\in\R^n$
$$\lu_1(x)<\lu(x). $$
We let
$$\tilde{\lu}_1(x,t)=A(t)\lu_1\left(\frac{x}{A(t)}\right), \ \ \tilde{\lu}(x,t)=A(t)\lu\left(\frac{x}{A(t)}\right), \ \ \text{and}\ \  \tilde{\uu}(x,t)=A(t)u_0\left(\frac{x}{A(t)}\right),$$
where $A(t)=[(1+\al)\td{t}]^{\frac{1}{1+\al}}\geq 1.$
Moreover, for any compact convex domain $K$ and positive constant $T$, let
$$2\delta=\min_{K\times[0,T]}\left(\tilde{\lu}-\tilde{\lu}_1\right).$$
 We define a spacelike function $\Psi=\tilde{\lu}_1+\delta$. Denote $\Omega=\{(x,t)\in\mathbb{R}^n\times [0,T]; \Psi\leq \tilde{\uu}\}$. It is clear that $K\times[0,T]\subset\Omega$. Since as $|x|\goto\infty,$ $\tilde{\lu}_1-\tilde{\uu}\goto 0$, we know that $\Omega$ is a compact set only depending on $K$ and $T$. Applying Lemma 5.1 of \cite{BS09},
if $\Omega\subset\bigcup_{t\in[0,T]}\Omega_r(t),$ then we have the gradient estimate:
\[\sup_{K\times[0,T]}\frac{1}{\sqrt{1-|Du_r|^2}}\leq\frac{1}{\delta}\sup\limits_{\Omega}\frac{\tilde{\bar{u}}-\Psi}{\sqrt{1-|D\Psi|^2}}.\]

\subsection{Local estimates for $F$}
 \label{sub-loc-F}Due to the complication of the local $C^2$ estimates, in this subsection we need to establish the local estimates for $F.$ More precisely, we want to bound $F$ from above and below in terms of $u=-\lt<X, e_{n+1}\rt>$ and $v=-\lt<\nu, e_{n+1}\rt>.$ By Subsection \ref{sub-loc-c1}, we know, this is equivalent to bounding $F$ from above and below by the height function $u.$
For our convenience, we will denote $\Phi=F^\al,$ where $F=\s_k^{\frac{1}{k}}$.
A straightforward calculation yields the following Lemma.
\begin{lemm}\label{evolution equation lem}
Under the flow \eqref{int1} we have
\be\label{eel1}
\mathfrak L u=(1-\al)\Phi v,
\ee
\be\label{eel2}
\mathfrak L v=-\sum\limits_i\Phi^{ii}\ka_i^2 v,
\ee
\be\label{eel3}
\mathfrak L \Phi=-\sum\limits_i\Phi^{ii}\ka_i^2\Phi,
\ee
and
\be\label{eel4}
\mathfrak L h_i^j=(\al-1)\Phi\sum \limits_k h_i^kh_k^j-h_i^j\sum\limits_{k, l, s}\Phi^{kl}h_k^sh_s^l+\sum\limits_{p, q, r, s}\Phi^{pq, rs}\bn_i h_{pq}\bn^j h_{rs}.
\ee
Here, $\mathfrak L=\p_t-\Phi^{ij}\bn_{ij}.$
\end{lemm}

Recall Lemma \ref{sap-lem3.1}, since $\tilde{F}=\Phi^{-1}$  we have $\Phi< C_2 v.$ This leads to a local upper bound of $F$ directly.

In order to obtain the local lower bound of $F$, we consider the following test function,
$$\varphi=\gamma\log(c-t-u)-\log\Phi+Av,$$
where $\gamma,A$ are two positive constants to be determined later and $c>0$ is an arbitrary positive constant.
At the maximal value point $(x_0,t_0)$ of $\varphi$, we have
\begin{eqnarray}\label{eqq1}
0=\varphi_i=\gamma\frac{-u_i}{c-t-u}-\frac{\Phi_i}{\Phi}+Av_i,
\end{eqnarray}
and
\begin{eqnarray}
0\geq\varphi_{ii}=\gamma\frac{-u_{ii}}{c-t-u}-\gamma\frac{u_i^2}{(c-t-u)^2}-\frac{\Phi_{ii}}{\Phi}+\frac{\Phi_i^2}{\Phi^2}+Av_{ii}.\nonumber
\end{eqnarray}
Therefore, we get
\begin{eqnarray}\label{eqq2}
\mathfrak{L}(\varphi)&=&\gamma\frac{-1-\mathfrak{L}u}{c-t-u}+\gamma\frac{\Phi^{ii}u_i^2}{(c-t-u)^2}-\frac{\mathfrak{L}\Phi}{\Phi}-\Phi^{ii}\frac{\Phi_i^2}{\Phi^2}+A\mathcal{L}v\\
&=&\gamma\frac{-1+(\al-1)\Phi v}{c-t-u}+\gamma\frac{\Phi^{ii}u_i^2}{(c-t-u)^2}-\Phi^{ii}\frac{\Phi_i^2}{\Phi^2}-(Av-1)\Phi^{ii}\kappa_i^2.\nonumber
\end{eqnarray}
By \eqref{eqq1}, we have
$$\gamma\frac{\Phi^{ii}u_i^2}{(c-t-u)^2}\leq \frac{2}{\gamma}\frac{\Phi^{ii}\Phi_i^2}{\Phi^2}+\frac{2A^2}{\gamma}\Phi^{ii}\kappa_i^2u_i^2\leq \frac{2}{\gamma}\frac{\Phi^{ii}\Phi_i^2}{\Phi^2}+\frac{2A^2v^2}{\gamma}\Phi^{ii}\kappa_i^2,$$
where we have used $\sum\limits_iu_i^2=v^2-1<v^2.$
Combining with \eqref{eqq2}, we obtain at $(x_0, t_0)$
\begin{eqnarray}\label{eqq3}
0&\leq&\gamma\frac{-1+(\al-1)\Phi v}{c-t-u}-\left(Av-1-\frac{2A^2v^2}{\gamma}\right)\Phi^{ii}\kappa_i^2-\left(1-\frac{2}{\gamma}\right)\frac{\Phi^{ii}\Phi_{i}^2}{\Phi^2}\nonumber.
\end{eqnarray}
When $A=2$ and $\gamma$ is chosen so large that
$$\gamma\geq4+2A^2v^2, $$ we get
$$(\al-1)\Phi v\geq 1.$$
This leads to the desired estimate.
We conclude
\begin{lemm}
\label{loc-F-lem}
Let $u_r^*$ be the solution of \eqref{sap1.1} and $u_r$ be the Legendre transform of $u_r^*.$ For any $c>0,$ denote $K:=\{(x, t)\mid u_r(x, t)+t\leq c\}$
and $V_0:=\max\limits_{(x, t)\in K}v.$ Then we have
\[\lt(\frac{c-t-u}{c}\rt)^\gamma e^{2(v-V_0)}\frac{1}{(\al-1)V_0}\leq\Phi<C_2V_0,\]
where $C_2=C_2(u_0^*)$ is determined by Lemma \ref{sap-lem3.1} and $\gamma=4+8V_0^2.$ We note that here $c$ is always chosen such that
$K\subset\bigcup_{t\in[0, \infty)} \lt(Du_r^*(B_r, t)\times\{t\}\rt).$
\end{lemm}

\subsection{Local $C^2$ estimates}
\label{lc2}
In this subsection, we will establish the local $C^2$ estimate for $u_r.$
\begin{lemm}
\label{loc-c2-lem}
Let $u_r^*$ be the solution of \eqref{sap1.1} and $u_r$ be the Legendre transform of $u_r^*.$ Denote $\Omega_r(t):=Du_r^*(B_r, t).$ For any given $c>0,$ let $r_c\in (0, 1)$ such that when $r>r_c,$ $u_r(\cdot, t)|_{\p\Omega_r(t)}>c$ for all $t\in[0, \infty).$ Then for $r>r_c$ we have
\[(c-u_r)^m\log\ka_{\max}(x, t)\leq C,\]
where $\ka_{\max}(x, t)$ is the largest principal curvature of $\M_{u_r}$ at $(x, t),$
$m$ is a large constant only depending on $k,$ and $C:=C(\Phi, c)>0$ is independent of $r.$
\end{lemm}
\begin{proof}
 In this proof, we will drop the subscript $r$ and denote $u_r$ by $u.$ Consider $$\varphi=\frac{(c-u)^m\log P_m}{1-\frac{m\Phi}{M}},$$ where $m$ is some positive integer to be determined later, $M\geq 2m\sup\limits_{u\leq c}\Phi,$ and
$$P_m=\sum_i\kappa_i^m.$$
Then we get
\[\log\varphi=m\log(c-u)+\log\log P_m-\log\lt(1-\frac{m\Phi}{M}\rt).\]
We suppose $\varphi$ achieves its maximum value at $(x_0,t_0).$ We may choose a local orthonormal frame $\{\tau_1,\cdots,\tau_n\}$  such that at $(x_0,t_0),$ $h_{ij}=\ka_k\delta_{ij}$ and $\kappa_1\geq\cdots\geq \kappa_n$.
At the maximal value point $(x_0, t_0),$ differentiating $\log\varphi$ twice we get
\be\label{lc2.5}
0=\frac{\varphi_i}{m\varphi}=\frac{-u_i}{c-u}+\frac{1}{\log P_m}\frac{\sum\limits_j\ka_j^{m-1}h_{jji}}{P_m}+\frac{\Phi_i}{M-m\Phi},
\ee
and
\be\label{lc2.6}
\begin{aligned}
0\geq&\frac{\varphi_{ii}}{m\varphi}=\frac{-u_{ii}}{c-u}-\frac{u_i^2}{(c-u)^2}-\frac{m}{(\log P_m)^2}\lt(\frac{\sum\limits_j\ka_j^{m-1}h_{jji}}{P_m}\rt)^2\\
&+\frac{1}{P_m\log P_m}\lt[\sum\limits_j\ka_j^{m-1}h_{jjii}+(m-1)\sum\limits_j\ka_j^{m-2}h_{jji}^2+\sum\limits_{p\neq q}\frac{\ka_p^{m-1}
-\ka_q^{m-1}}{\ka_p-\ka_q}h^2_{pqi}\rt]\\
&-\frac{m}{P_m^2\log P_m}\lt(\sum\limits_j\ka_j^{m-1}h_{jji}\rt)^2+\frac{\Phi_{ii}}{M-m\Phi}+\frac{m(\Phi)^2_i}{(M-m\Phi)^2}.
\end{aligned}
\ee
Therefore, at $(x_0, t_0)$ we obtain
\be\label{lc2.7}
\begin{aligned}
0\leq&\frac{\mathfrak{L}\varphi}{m\varphi}=\frac{-\mathfrak{L}u}{c-u}+\frac{1}{\log P_m}\frac{\sum\limits_j\ka_j^{m-1}\mathcal{L}\ka_j}{P_m}\\
&+\frac{\mathfrak{L}\Phi}{M-m\Phi}+\frac{\Phi^{ii}u_i^2}{(c-u)^2}+\frac{m\Phi^{ii}}{(\log P_m)^2}\lt(\frac{\sum\limits_j\ka_j^{m-1}h_{jji}}{P_m}\rt)^2\\
&-\frac{\Phi^{ii}}{(\log P_m)P_m}\lt[(m-1)\sum\limits_j\ka_j^{m-2}h^2_{jji}+\sum\limits_{p\neq q}\frac{\ka_p^{m-1}-\ka_q^{m-1}}{\ka_p-\ka_q}h^2_{pqi}\rt]\\
&+\frac{m\Phi^{ii}}{\log P_m}\lt(\frac{\sum\limits_j\ka_j^{m-1}h_{jji}}{P_m}\rt)^2-\frac{m\Phi^{ii}(\Phi_i)^2}{(M-m\Phi)^2}.
\end{aligned}
\ee
For our convenience, in below, we will denote $\beta=\frac{\al}{k},$ then $\Phi=\s_k^\beta.$
Plugging
\[\Phi^{pq}=\beta\s_k^{\beta-1}\s_k^{pq}\text{ and }\ \
\Phi^{pq, rs}=\beta\s_k^{\beta-1}\s_k^{pq, rs}+\beta(\beta-1)\s_k^{\beta-2}\s_k^{pq}\s_k^{rs}\]
into \eqref{lc2.7}, we get
\be\label{lc2.8}
\begin{aligned}
0\leq&\frac{(\beta k-1)\Phi v}{c-u}+\frac{1}{(\log P_m)P_m}\sum\limits_j\ka_j^{m-1}\left[(\beta k-1)\Phi\ka_j^2-\ka_j\Phi^{ii}\ka_i^2\right.\\
&\left.+\beta\s_k^{\beta-1}\s_k^{pq, rs}h_{pqj}h_{rsj}+\beta(\beta-1)\s_k^{\beta-2}(\s_k)_j^2\right]\\
&-\frac{\Phi\Phi^{ii}\ka_i^2}{M-m\Phi}+\frac{\beta\s_k^{\beta-1}\s_k^{ii}u^2_i}{(c-u)^2}+\frac{m\beta\s_k^{\beta-1}\s_k^{ii}}{(\log P_m)^2}\lt(\frac{\sum\limits_j\ka_j^{m-1}h_{jji}}{P_m}\rt)^2\\
&-\frac{\beta\s_k^{\beta-1}\s_k^{ii}}{(\log P_m)P_m}\lt[(m-1)\sum\limits_j\ka_j^{m-2}h^2_{jji}+\sum\limits_{p\neq q}\frac{\ka_p^{m-1}-\ka_q^{m-1}}{\ka_p-\ka_q}h^2_{pqi}\rt]\\
&+\frac{m\beta\s_k^{\beta-1}\s_k^{ii}}{\log P_m}\lt(\frac{\sum\limits_j\ka_j^{m-1}h_{jji}}{P_m}\rt)^2-\frac{m\beta\s_k^{\beta-1}\s_k^{ii}(\Phi_i)^2}{(M-m\Phi)^2}.
\end{aligned}
\ee
For any fixed index $1\leq i\leq n$, we denote
\[A_i=\frac{\ka_i^{m-1}}{P_m}\lt[K(\s_k)_i^2-\sum\limits_{p, q}\s_k^{pp, qq}h_{ppi}h_{qqi}\rt], \ \ B_i=\frac{2}{P_m}\sum\limits_j\ka_j^{m-1}\s_k^{jj, ii}h^2_{jji},\]
\[C_i=\frac{m-1}{P_m}\s_k^{ii}\sum\limits_j\ka^{m-2}_jh_{jji}^2,\ \ D_i=\frac{2\s_k^{jj}}{P_m}\sum\limits_{j\neq i}\frac{\ka_j^{m-1}-\ka_i^{m-1}}{\ka_j-\ka_i}h^2_{jji},\]
and
\[E_i=\frac{m}{P_m^2}\s_k^{ii}\lt(\sum\limits_j\ka_j^{m-1}h_{jji}\rt)^2,\]
where $K>0$ depends on $\min\limits_{u\leq c}\s_k.$\
Then \eqref{lc2.8} becomes
\be\label{lc2.10}
\begin{aligned}
0&\leq\frac{(\beta k-1)\Phi v}{c-u}+\frac{n(\beta k-1)\Phi\ka_1}{\log P_m}-\frac{\beta\s_k^{\beta-1}\s_k^{ii}\ka_i^2}{\log P_m}\\
&-\frac{\beta\s_k^{\beta-1}}{\log P_m}\sum_i\lt[A_i+B_i+C_i+D_i-\lt(1+\frac{1}{\log P_m}\rt)E_i\rt]\\
&+\frac{\beta\s_k^{\beta-1}}{\log P_m}\sum\limits_i\frac{\ka_i^{m-1}}{P_m}K(\s_k)_i^2+\frac{\beta(\beta-1)\s_k^{\beta-2}}{(\log P_m)P_m}
\sum\limits_j\ka_j^{m-1}(\s_k)^2_j\\
&-\frac{\Phi\Phi^{ii}\ka_i^2}{M-m\Phi}+\frac{\beta\s_k^{\beta-1}\s_k^{ii}u_i^2}{(c-u)^2}
-\frac{m\beta\s_k^{\beta-1}\s_k^{ii}(\Phi_i)^2}{(M-m\Phi)^2}.
\end{aligned}
\ee
Moreover, it is easy to see that $\s_k^{11}\ka_1\geq \eta_0\s_k$ for some $\eta_0=\eta_0(n, k),$ which implies $\sum\s_k^{ii}\geq\frac{\eta_0\s_k}{\ka_1}.$
If $\log\ka_1>c_1M^2$ for $c_1=c_1(\s_k, \eta_0, m, K, \beta)$, then we have
\be\label{lc2.11}
\frac{m}{2M^2}\s_k^{ii}[\al\s_k^{\beta-1}(\s_k)_i]^2>\frac{\ka_i^{m-1}}{P_m\log P_m}K(\s_k)_i^2
+\frac{|\beta-1|\s_k^{-1}\ka_i^{m-1}(\s_k)_i^2}{(\log P_m)P_m}.
\ee
Note by \eqref{lc2.5} we obtain
\be\label{lc2.12}
\begin{aligned}
\lt(\frac{u_i}{c-u}\rt)^2&=\lt(\frac{1}{\log P_m}\frac{\sum\limits_j\ka_j^{m-1}h_{jji}}{P_m}+\frac{\Phi_i}{M-m\Phi}\rt)^2\\
&\leq\frac{2}{(\log P_m)^2}\lt(\frac{\sum\limits_j\ka_j^{m-1}h_{jji}}{P_m}\rt)^2+\frac{2(\Phi_i)^2}{(M-m\Phi)^2}.
\end{aligned}
\ee
Combining \eqref{lc2.11} and \eqref{lc2.12} with \eqref{lc2.10}, we get
\be\label{lc2.13}
\begin{aligned}
0&\leq\frac{(\al-1)\Phi v}{c-u}+\frac{n(\al-1)\Phi\ka_1}{\log P_m}-\frac{\beta\s_k^{\beta-1}\s_k^{ii}\ka_i^2}{\log P_m}\\
&-\frac{\beta\s_k^{\beta-1}}{\log P_m}\sum_i\lt[A_i+B_i+C_i+D_i-\lt(1+\frac{1}{\log P_m}+\frac{2}{m\log P_m}\rt)E_i\rt]\\
&-\frac{\Phi\Phi^{ii}\ka_i^2}{M-m\Phi}-\lt(\frac{m}{2}-2\rt)\frac{\beta\s_k^{\beta-1}\s_k^{ii}(\Phi_i)^2}{(M-m\Phi)^2}.
\end{aligned}
\ee
If we assume $m, K$ and $\kappa_1$ are all sufficiently large, the Lemma 8 and Lemma 9 of \cite{LRW16} gives
$$\sum_i\lt[A_i+B_i+C_i+D_i-\lt(1+\frac{1}{\log P_m}+\frac{2}{m\log P_m}\rt)E_i\rt]\geq 0.$$
It's clear that when $\log\ka_1>CM$ for some $C=C(n, k, \s_k, m, \beta),$ we have
\[\frac{\Phi\Phi^{ii}\ka_i^2}{2M}>\frac{n(\al-1)\Phi\ka_1}{\log P_m}.\]
Thus, \eqref{lc2.13} yields
\[0\leq\frac{(\al-1)\Phi v}{c-u}-\frac{\eta_0\beta\Phi^2\ka_1}{2(M-m\Phi)},\] which gives the desired estimate.
\end{proof}

\section{Convergence}
\label{conv}
In Section \ref{sap}, we have shown there exists a solution to the initial value problem \eqref{sap1.1}.
Now, denote \be\label{urs}\td{u}^*_r(x,t)=\frac{u_r^*(x,t)}{A(t)}, \text{ where } \ \ A(t)=[(1+\al)\td{t}]^\frac{1}{1+\al}=[(1+\al)t+1]^\frac{1}{1+\al}.\ee
Let $\tau=\int_0^t[(1+\al)s+1]^{-1}ds,$
then $\td{u}^*_r$ satisfies
\be\label{conv1.2}
\left\{
\begin{aligned}
(\td{u}^*_r)_{\tau}&=-F_*^{-\al}(w^*\gas_{ik}(\td{u}^*_r)_{kl}\gas_{lj})w^*-\td{u}^*_r\,\,&\mbox{in $B_r\times(0, T]$},\\
\td{u}^*_r(\cdot, t)&=u_0^*\,\, &\mbox{on $\p B_r\times[0, T],$}\\
\td{u}^*_r(\cdot, 0)&=u^*_0\,\,&\mbox{on $B_r\times\{0\}.$}
\end{aligned}
\right.
\ee
Notice that if $X_r=(x, u_r(x, t))$ for $(x, t)\in Du^*_r(B_r, t)\times\{t\},$ is the position vector for the graph $u_r$ which is the Legendre transform of
$u^*_r,$ then $\td{X}_r=\frac{X_r}{A(t)}=\lt(A(t)x, \frac{1}{A(t)}u_r(A(t)x, t)\rt),$ where $(x, t)\in\frac{1}{A(t)}Du_r^*(B_r, t)\times\{t\},$
is the position vector for the graph of the Legendre transform of $\td{u}^*_r.$
In the following we will prove two Lemmas.
\begin{lemm}
\label{conv-lem1}
Let $\td{u}_r^*$ be defined as in  \eqref{urs}, then we have $\td{u}^*_r(\cdot, t)\goto u_r^{\infty*}(\cdot)$ uniformly in $B_r$ as $t\goto\infty.$ Here $u^{\infty*}_r$
satisfies
\be\label{conv1.3}
\left\{
\begin{aligned}
F_*^{-\al}(w^*\ga^*_{ik}(u^{\infty*}_r)_{kl}\gamma_{lj}^*)w^*&=-u^{\infty*}_r\,\,&\mbox{in $B_r$}\\
u^{\infty*}_r&=u^*_0\,\,&\mbox{on $\p B_r$}.
\end{aligned}
\right.
\ee
\end{lemm}
The proof of this Lemma will be given in the Subsection 4.1.

\begin{lemm}
\label{conv-lem2}
Let $u_r$ be the Lengendre transform of $u_r^*$. Then $\frac{1}{A(t)}u_r(A(t)x, t)\goto\frac{1}{A(t)}u(A(t)x, t)$ as $r\goto 1$ uniformly in any compact subset of
$\R^n\times[0, \infty).$
\end{lemm}
The proof of this Lemma will be given in the Subsection 4.1.

In view of Section 3 of \cite{WX22-1} we can see that as $r\goto 1$, $u_r^{\infty}$ which is the Legendre transform of $u_r^{\infty*}$ converges to $u^\infty$ uniformly on any compact set $K\subset\R^n,$ and $u^\infty$ satisfies
\[
\left\{
\begin{aligned}
\s_k^{\frac{\al}{k}}(\ka[\M_{u^\infty}])&=-\lt<X_{u^\infty}, \nu_{u^\infty}\rt>\\
u^\infty-|x|&\goto\varphi\lt(\frac{x}{|x|}\rt).
\end{aligned}
\right.
\]
Combining this fact with Lemma \ref{conv-lem1} and Lemma \ref{conv-lem2}, we conclude
\begin{coro}
\label{conv-cor1}
Let $u^*$ be the solution of the initial value problem \eqref{int1.1} and $u$ be the Lengendre transform of $u^*.$ Then
for any sequence $\{t_j\}\goto\infty$ there exists a subsequence $\{t_{j_k}\}\goto\infty$ such that
\[\frac{1}{A(t_{j_k})}u(A(t_{j_k})x, t_{j_k})\goto u^\infty(x)\]
uniformly in any compact set $K\subset\R^n.$ Moreover, $u^\infty$ satisfies
\be\label{conv1.4}
\left\{
\begin{aligned}
\s_k^{\frac{\al}{k}}(\ka[\M_{u^\infty}])&=-\lt<X_{u^\infty}, \nu_{u^\infty}\rt>\\
u^\infty-|x|&\goto\varphi\lt(\frac{x}{|x|}\rt).
\end{aligned}
\right.
\ee
\end{coro}

\subsection{Proof of Lemma \ref{conv-lem1} and Lemma \ref{conv-lem2}}
\subsubsection{$C^0$ estimates for $\td{u}^*_r$}
By our assumption \eqref{Cond3} on the initial hypersurface we have
\[\s_k^{\frac{\al}{k}}(\ka[\M_{u_0}])<\frac{-u_0^*}{w^*}=-\lt<X_{u_0}, \nu_{u_0}\rt>,\]
which yields
\[F_*^{-\al}(w^*\ga^*_{ik}(u_0^*)_{kl}\ga_{lj}^*)w^*+u_0^*< 0.\]
Therefore, $u_0^*$ is a subsolution of \eqref{conv1.2}. It's clear that $\lu^{*}$ constructed in Subsection \ref{sap2}
is a supersolution of \eqref{conv1.2}. We conclude that
\[u_0^*\leq\td{u}^*_r\leq\lu^{*}.\]
Here, the $C^0$ estimate of $\td{u}^*_r$ is independent of $r$ and $t$.
\subsubsection{
$C^1$ estimate for $\td{u}^*_r$}
Consider
\be\label{conv1.3}
\left\{
\begin{aligned}
F_*(w^*\ga^*_{ik}u^*_{kl}\ga^*_{lj})&=\lt(\frac{-u^*}{w^*}\rt)^{-\frac{1}{\al}}\,\,&\mbox{in $B_r$}\\
u^*&=u_0^*\,\,&\mbox{on $\p B_r.$}
\end{aligned}
\right.
\ee
By Section 3 of \cite{WX22-1}, we know there exist a solution $\hat{u}^*_r$ of \eqref{conv1.3}.
In view of the standard maximum principle we have
$$u_0^*\leq\td{u}^*_r\leq\hat{u}^*_r\,\, \mbox{and $\td{u}^*_r=u_0^*=\hat{u}^*_r$ on $\p B_r\times[0, \infty).$}$$
This gives $|D\td{u}^*_r|\leq C,$ for some $C>0$ independent of $t.$
\subsubsection{
Bounds for $F_*(w^*\gas_{ik}(\td{u}_r^*)_{kl}\gas_{lj}):=\td{F}_*$}

Let's denote $\td{H}:=-\td{F}^{-\al}_*w^*-\td{u}^*_r,$ a straightforward calculation yields
\[\mathcal{L}\td{H}:=-\td{H},\]
where $\mathcal{L}:=\frac{\p}{\p\tau}-\al(w^*)^2\td{F}^{-\al-1}\td{F}^{ij}_*\gas_{ik}\gas_{lj}\p^2_{kl}.$
By the assumption \eqref{Cond3} on the initial surface we know at $\tau=0,$ $\td{H}\geq 0.$ Moreover, applying the standard short time existence Theorem,
we get on $\p B_r\times(0, \infty),$ $\td{H}=0.$ Therefore, we conclude that $\td{H}\geq 0$ in $\bar{B}_r\times [0, \infty),$
which yields $\td{F}_*\geq\lt(\frac{-\td{u}^*_r}{w^*}\rt)^{-\frac{1}{\al}}\geq C_1,$ where $C_1$ is independent of $t.$

On the other hand, recall Lemma \ref{sap-lem3.1} we know
\[F_*^\al<\frac{A(T)^\al}{C_3}\,\,\mbox{on $\bar{U}_r\times[0, T].$}\]
In particular, for any $T>0$ we get
\[F_*^\al(w^*\ga_{ik}^*\lt(u_r^*(\cdot, T))_{kl}\ga^*_{lj}\rt)<\frac{A(T)^\al}{C_3},\]
where $C_3$ depends on $\M_0$ and $r$.
This gives
\[F_*^\al(w^*\gas_{ik}(\td{u}_r^*)_{kl}\gas_{lj})A(t)^\al<\frac{A(t)^\al}{C_3},\]
which is equivalent to
\[\td{F}^\al_*<\frac{1}{C_3}.\]
Here, note that $C_3$ is independent of $t.$

\subsubsection{$C^2$ estimates for $\td{u}^*_r$} Now let $\td{v}_r=\frac{\td{u}^*_r}{w^*}$ then $\td{v}_r$ satisfies
\be\label{conv1.4}
\left\{
\begin{aligned}
\lt(\td{v}_r\rt)_\tau &=-F_*^{-\al}\lt(\td{\Lambda}_{ij}\rt)-\td{v}_r\,\,&\mbox{in $U_r\times(0, \infty),$}\\
\td{v}_r &=v^{*}_0=\frac{u_0^*}{w^*}\,\,&\mbox{on $\p U_r\times[0, \infty)$},\\
\td{v}_r&=v^{*}_0\,\,&\mbox{on $U_r\times\{0\},$}
\end{aligned}
\right.
\ee
where $\td{\Lambda}_{ij}=\bn_{ij}\td{v}_r-\td{v}_r\delta_{ij}.$
We will denote $\td{F}^{-1}=F_*^{-\al}(\td{\Lambda}_{ij}).$ For any smooth function $\phi,$ we define
\[L\phi:=\phi_\tau-\td{F}^{-2}\td{F}^{ij}_{\td{v}}\nabla_{ij}\phi
+\lt(\td{F}^{-2}_{\td{v}}\sum\limits_i\td{F}^{ii}_{\td{v}}+1\rt)\phi.\]
Notice that we have proved the $C^0, C^1$ estimates of $\td{u}_r^*$ are independent of $t.$ This implies the $C^0, C^1$ estimates of $\td{v}_r$ are independent of $t.$
Moreover, we also know the upper and lower bounds of $\td{F}^{-1}$ are independent of $t.$ By a small modification of the argument in Subsection \ref{sap5}, we obtain a $C^2$ boundary estimate of
$\td{v}_r$ that is independent of $t.$ The global $C^2$ estimate for $\td{v}_r$ follows from a small modification of the proof of Lemma 20 in \cite{WX212}, and it is not hard to see that this estimate is also independent of $t$.

\subsubsection{Proof of Lemma \ref{conv-lem1}}By the these uniform estimates of $\td{u}^*_r$, we conclude
$$\lim\limits_{\tau\goto\infty}\td{u}^*_r(\xi,\tau)=u^{\infty*}_r(\xi)$$
uniformly in $B_r.$   It's cleat that
\[\td{u}^*_r(x, \tau)-\td{u}_r^*(x, 0)=\int_0^\tau\td{H}ds.\]
Then the uniform $C^0$ bound for $\td{u}^*_r$ implies $\int_0^\infty\td{H}ds<\infty.$ This yields
as $\tau\goto\infty,$ $\td{H}\goto 0.$
Therefore,  $u^{\infty*}_r$ satisfies
\[
\left\{
\begin{aligned}
F_*^{-\al}(w^*\ga^*_{ik}(u^{\infty*}_r)_{kl}\gamma_{lj}^*)w^*&=-u^{\infty*}_r\,\,&\mbox{in $B_r$}\\
u^{\infty*}_r&=u^*_0\,\,&\mbox{on $\p B_r$}.
\end{aligned}
\right.
\]\qed

\subsubsection{Proof of Lemma \ref{conv-lem2}}
We want to show that for any $K\subset\R^n,$ there exists $r_K>0,$ such that when $r>r_K,$ $\frac{1}{A(t)}u_r(A(t)x, t)$ is defined in $K$ for any $t>0.$
We denote  $\tilde{\Omega}(r,t)=D\tilde{u}^*_r(B_r,t)$ and $\hat{\Omega}(r)=D\hat{u}_r^*(B_r),$ where $\hat{u}_r^*$ is the solution of \eqref{conv1.3}. In the following, we only need to show if $r>r_K$ then $K\subset\td{\Omega}(r, t)$ for any $t>0.$ It's clear that
$$\tilde{u}_r|_{\p\tilde{\Omega}(r,t)}=(\xi\cdot D\td{u}^*_r-\td{u}^*_r)|_{\p B_r}\geq (\xi\cdot D\hat{u}_r^*-\hat{u}_r^*)|_{\p B_r}\geq\hat{u}_r|_{\p \hat\Omega(r)},$$
where $\hat{u}_r$ is the Legendre transform of $\hat{u}_r^*.$
Recall \cite{WX22-1} we know when $r\goto 1$, $\hat{u}_r |_{\p \hat\Omega(r)}\goto +\infty$, which yields as $r\goto 1,$ $\tilde{u}_r|_{\p\tilde{\Omega}(r,t)}\goto+\infty$.
By virtue of Subsubsection 4.1.1 we know $\td{u}_r>c$ for some constant $c$ independent of $r$ and $t.$ Since $|D\td u_r|<1$ , we get $\td u_r$ is uniformly bounded from above in $K.$ We conclude that for any compact set $K\subset\R^n,$ there exists $r_K>0$ such that when $r>r_K,$ $K\subset\td{\Omega}(r, t)$ for any $t>0.$

Combining the discussion above with estimates obtained in Section \ref{le}, it is easy to see that $u_r$ convergence to $u$ on any compact subset $K\times [a, b]\subset\R^n\times [0, \infty),$ $0\leq a<b.$ Therefore, we prove Lemma \ref{conv-lem2}. \qed

\section*{Appendix}
\subsection*{Proof of the equivalence of \eqref{Cond3} and \eqref{Cond}}

When $C>1,$ we consider
$\hat{u}_0(x)=\beta u_0\lt(\frac{x}{\beta}\rt),$ where $\beta$ is some undetermined  positive constant. A straightforward calculation yileds
\[F^\al(\ka[\M_{\hat{u}_0}])=\beta^{-\al}F^\al(\ka[\M_{u_0}]),\]
and
\[-\lt<X_{\hat{u}_0}, \nu_{\hat{u}_0}\rt>=\frac{-x\cdot Du_0\lt(\frac{x}{\beta}\rt)
+\beta u_0\lt(\frac{x}{\beta}\rt)}{\sqrt{1-|Du_0\lt(\frac{x}{\beta}\rt)|^2}}=-\beta\lt<X_{u_0}, \nu_{u_0}\rt>.\]
Therefore, we can see that $\hat{u}_0$ satisfies
\[F^\al(\ka[\M_{\hat{u}_0}])<C\beta^{-\al-1}(-\lt<X_{u_0}, \nu_{u_0}\rt>).\]
Choosing $\beta>0$ sufficiently large, then $\hat{u}_0$ satisfies \eqref{Cond}. It is easy to see that $\hat{u}_0$ is also spacelike, strictly convex, and
$$\hat{u}_0(x)-|x|\goto\beta\varphi\lt(\frac{x}{|x|}\rt),\text{ as }\ \ |x|\goto\infty.$$
Applying Theorem \ref{theo1}, we know there exists $\hat{u}(x, t)$ such that
\[
\left\{
\begin{aligned}
\hat{u}_t&=&F^\al w\,\,&\mbox{ in $\R^n\times(0, \infty),$}\\
\hat{u}(x, 0)&=&\hat{u}_0(x)\,\,&\mbox{ in }\R^n.
\end{aligned}
\right.
\]
Moreover, the rescaled flow $\lt(A(t)x, \frac{\hat{u}(A(t)x, t)}{A(t)}\rt)$ converges to a self-expander $\M_{\hat{u}^{\infty}}:=\{(x, \hat u^\infty(x))\mid x\in\R^n\}$ with
$$\hat{u}^{\infty}(x)-|x|\goto\beta\varphi\lt(\frac{x}{|x|}\rt),\text{ as }\ \ |x|\goto\infty.$$
Now, let $u(x, t)=\frac{1}{\beta}\hat{u}(\beta x, \beta^{\al+1}t).$ One can verify that $u$ satisfies
\[
\left\{
\begin{aligned}
u_t&=&F^\al w\,\,&\mbox{  in $\R^n\times(0, \infty),$}\\
u(x, 0)&=&u_0(x)\,\,&\mbox{  in } \mathbb{R}^n.
\end{aligned}
\right.
\]
Moreover, the rescaled flow $\lt(A(t)x, \frac{u(A(t)x, t)}{A(t)}\rt)$ converges to the self-expander $\M_{u^{\infty}}:=\{(x, u^\infty(x))\mid x\in\R^n\}$ with
$$u^{\infty}(x)-|x|\goto\varphi\lt(\frac{x}{|x|}\rt),\text{ as }\ \ |x|\goto\infty.$$
Here $u^\infty(x)=\frac{1}{\beta}\hat u^\infty(\beta x).$ \qed

\end{document}